\documentclass[11pt]{amsart}
\usepackage{geometry,graphicx}
\usepackage{latexsym}
\usepackage{labelfig}
\usepackage{amssymb}
\usepackage[cp850]{inputenc}
\usepackage{epsfig}
\usepackage{epstopdf}
\usepackage{psfrag}
\usepackage{amsthm}
\usepackage{amscd}
\usepackage{amsmath}
\usepackage{amsfonts}
\usepackage{graphics}
\usepackage{enumerate}

\theoremstyle{theorem}

\newtheorem{theorem}{Theorem}

\newtheorem{proposition}{Proposition}

\newtheorem{lemma}{Lemma}

\newtheorem{corollary}{Corollary}

\newtheorem{conjecture}{Conjecture}

\newtheorem{property}{Property}

\theoremstyle{definition}

\newtheorem{definition}{Definition}

\newtheorem*{remark}{Remark}

\newcommand{\sys}{{\rm sys}}

\newcommand{\R}{{\mathbb R}}

\newcommand{\Z}{{\mathbb Z}}

\newcommand{\BB}{{\mathcal B}}

\newcommand{\arc}{{\mathcal A}}

\newcommand{\T}{{\mathcal T}}

\newcommand{\MM}{{\mathcal M}}

\newcommand{\area}{{\rm area}}

\newcommand{\arcsinh}{{\,\rm arcsinh}}

\newcommand{\arccosh}{{\,\rm arccosh}}

\newcommand{\PP}{\mathcal P}

\newcommand{\FF}{\mathcal F}

\newcommand{\II}{\mathcal I}

\numberwithin{equation}{section}

\begin{document}

%




%

\title{Bers' constants for punctured spheres and hyperelliptic surfaces}




\author[F.~Balacheff]{Florent Balacheff}

\address[Florent Balacheff] {Laboratoire Paul Painlev\'e, Universit\'e Lille 1\\ Lille, France}

\email{florent.balacheff@math.univ-lille1.fr}

\author[H.~Parlier]{Hugo Parlier}
\address[Hugo Parlier]{Department of Mathematics, University of Toronto\\
  Toronto, Canada}
\email{hugo.parlier@gmail.com}


%

\date{\today}



%

\begin{abstract} This article is dedicated to proving Buser's conjecture about Bers' constants for spheres with cusps (or marked points) and for hyperelliptic surfaces. More specifically, our main theorem states that any hyperbolic sphere with $n$ cusps has a pants decomposition with all of its geodesics of length bounded by a constant roughly square root of $n$. Other results include lower and upper bounds for Bers' constants for hyperelliptic surfaces and spheres with boundary geodesics.\end{abstract}




%

\subjclass[2000]{Primary Secondary }


\maketitle

\section{Introduction}

Consider a hyperbolic surface of genus $g$ with $n$ cusps. It admits many pants decompositions: collections of simple closed geodesics whose complementary region consists of a collection of surfaces which are topologically thrice punctured spheres. Even up to homeomorphism, there are roughly $g!$ different pants decompositions of a genus $g$ surface. Among these, Lipman Bers \cite{be74,be85} showed that there is always one with all of its geodesics of length bounded by a constant which only depends on the topology of the surface. These constants are called Bers' constants.\\

One way of thinking of this is as a generalization of the fact that the length of the shortest closed geodesic (or systole) of a surface can be bounded by a constant which only depends on genus. There is a certain parallel between the two problems, although there are certain differences we shall outline. Both give a collection of constants and in both cases their behavior has been studied quite closely. For systoles, a great deal of attention has been given to the study of surfaces which are either the global maxima (for a fixed topology) or local extrema of the systole function \cite{sc931, sc98} and in fact the systole function is a topological morse function \cite{ak03}. However, for closed surfaces of genus $g$ the global maximum is only known for genus $2$ \cite{je84} and similarly, the only known Bers' constant for surfaces of closed genus is also in genus $2$ \cite{gen08}.\\

In both cases, quite a bit of effort has been put into studying the behavior of the constants as the topology grows. For closed hyperbolic surfaces, it is straightforward to find an upper-bound on systole length which grows logarithmically in genus. A lower bound that also grows logarithmically in genus is far less obvious \cite{busa94}. So, the rough asymptotic behavior is known for systoles and the question of what (or if there is) asymptotic behavior is wide open. For Bers' constants, less is known. Bers' original argument was an existence result and although there were arguments that yielded computable constants \cite{abbook}, the first real attempts at finding good upper-bounds are due to Peter Buser \cite{buhab}. These results were later improved by Buser and Buser-Seppal\"a \cite{buse92,bubook} and ultimately led to upper-bounds that grow linearly (in genus for closed surfaces, and in Euler characteristic for surfaces with cusps). Interesting lower bounds are also due to Buser \cite{buhab,bubook} and the lower bounds grow like square root (of Euler characteristic). Buser conjectures the following.

\begin{conjecture}\label{con:squareroot}
Bers' constants for surfaces of genus $g$ with $n$ cusps behave roughly like $\sqrt{g+n}$.
\end{conjecture}

Interestingly there is a striking difference between the two problems. For systoles, a rough upper bound is immediate and the lower bound is given by families of surfaces coming from arithmetic constructions \cite{busa94,kascvi07}. For Bers' constants, if the conjecture is correct, the lower bound, although clever, does not require any outside technology from number theory. And the upper bound seems to far from immediate.\\

Bers' constants have been used in different ways for studying the geometry of surfaces and Teichm\"uller surfaces. Some of the original motivations included bounds on the number of isospectral non-isometric surfaces, but also for finding rough fundamental domains for the mapping class group \cite{bubook}. More recently, Brock \cite{br03} showed that the application that takes a surface, and sends it the set of curves that compose one of its short pants decompositions (that we know exists by Bers' theorem) provides a quasi-isometry between Teichm\"uller space with the Weil-Peterson metric and the pants graph.\\

The main goal of this article is to prove Buser's conjecture in the case of punctured spheres (theorem \ref{th:berspuncturedsphere}) and hyperelliptic surfaces (theorem \ref{th:bershypsurf}). Although we provide lower bounds which improve the known bounds, the upper bounds are the real novelty and are closely related to recent works of the two authors with St\'ephane Sabourau, see \cite{basadias} and \cite{bps}. Indeed,  as observed in \cite[Proposition 6.3]{bps}, we can recover an optimal square root upper-bound on Bers' constants for puncture growth and an alternative proof of Buser's linear bounds for genus growth using the main result of~\cite{basadias}. But the approach used in those articles isn't well adapted to the context of Buser's conjecture, and the constants thus obtained are huge. In this paper, we prove Buser's conjecture in the case of punctured spheres by showing that it fundamentally relies on a classical result: the so-called Besicovitch lemma (see section \ref{sec:mainlemmas}). Then we prove Buser's conjecture in the case of hyperelliptic surfaces. Although the ideas work very nicely to prove upper bounds for punctured or hyperelliptic surfaces, they do not appear to extend easily to solving the genus case. \\

The organization of the article is as follows. Section \ref{sec:prelim} is dedicated to preliminaries which also includes the proof of some general properties that Bers' constants satisfy. In section \ref{sec:mainlemmas} we show the main lemmas used in our proof and as they apply to different cases, we put them in a separate section. In section \ref{sec:cusps} we show the upper bounds for spheres with cusps. In section \ref{sec:cone} we show the corresponding result for spheres with cone singularities of angle $\pi$, which in turn is the main ingredient for the upper bound for hyperelliptic surfaces (section \ref{sec:hyper}). In section \ref{sec:boundary}, we apply our methods to obtain bounds on Bers' constants for surfaces with geodesic boundary. Lower bounds for all of the above Bers' constants are provided in section \ref{sec:ex}.\\

\noindent {\bf Acknowledgements.} The authors are grateful to St\'ephane Sabourau for many interesting conversations without which this project would not have been possible. The authors are also grateful to Greg McShane for suggesting to concentrate on the punctured sphere case. Finally, the authors thank Peter Buser for his encouragement.

\section{Notations and Preliminaries}\label{sec:prelim}

Let $g,n$ be positive integers such that $2-2g - n$ is negative. The Teichm\"uller space of hyperbolic surfaces of genus $g$ with $n$ cusps will be denoted $\T_{g,n}$. The set of surfaces of genus $g$ with $n$ cusps up to isometry will be called {\it moduli space} and denoted $\MM_{g,n}$. So in our setting $\MM_{g,n}$ can be obtained as the quotient of $\T_{g,n}$ by the (extended) mapping class group.\\

A general tool that we will be using regularly is provided by the following lemma, which we shall call the length expansion lemma \cite{pa051,thspine}.

\begin{lemma} Let $S$ be a surface with $n > 0$ boundary curves $\gamma_1,\hdots,\gamma_n$. For
$(\varepsilon_1,\hdots,\varepsilon_n) \in (\R^+)^n$ with at least
one $\varepsilon_i\neq 0$, there exists a surface $\tilde{S}$ with
boundary geodesics of length
$\ell(\gamma_1)+\varepsilon_1,\hdots,\ell(\gamma_n)+\varepsilon_n$ such that
all corresponding simple closed geodesics in $\tilde{S}$ are of
length strictly greater than those of $S$.
\end{lemma}

For a given surface $S\in \MM_{g,n}$ and a pants decomposition $\PP$ of $S$, we define the length of $\PP$ as 
$$
\ell(\PP)=\max_{\gamma \in \PP} \ell(\gamma).
$$
The {\it Bers' constant} of $S$, denoted $\BB(S)$, is then the length of a shortest pants decomposition of the given surface $S$. The quantity $\BB_{g,n}$ is defined as
$$
\BB_{g,n}=\sup_{S\in \MM_{g,n}} \ell(\BB(S)).
$$
This quantity is well defined by Bers' original theorem \cite{be74} as this quantity can be bounded by a function that only depends on $g$ and $n$. Explicit bounds were first calculated in \cite{abbook}. Buser's investigations led to a number of bounds \cite{buhab,buse92}, where best lower and upper bound for closed surfaces of genus $g$ can be found in \cite[Theorems 5.1.3, 5.1.4]{bubook}.

\begin{theorem}\label{thm:buser}
Bers' constants satisfy $\sqrt{6g}-2 \leq \BB_{g,0} \leq 6 \sqrt{3\pi} (g-1)$.
\end{theorem}

We now begin by showing a certain number of properties that these constants have. In the following result, we show that in the definition of $\BB_{g,n}$, the ``$\sup$" can be made a ``$\max$". This may be well known to specialists, but in fact we can get some control on how thin a maximal surface can be. 

\begin{property}\label{property:maxbers1}
There exists a surface $S_{\max}\in \MM_{g,n}$ such that $\BB(S_{\max}) = \BB_{g,n}$. Furthermore $\sys(S_{\max})\geq s_{g,n}$ where $s_{g,n}>0$ is a constant that only depends on $g$ and $n$.
\end{property}

\begin{proof}
Given a surface $S \in \MM_{g,n}$, by the collar theorem (\cite{kee74}, and see \cite[Theorem 4.4.6]{bubook} for the version we use) any simple closed geodesic that crosses a geodesic of length $\ell$ has length at least $2\,\arcsinh{1\over \sinh\frac{\ell}{2}}$ . 
We consider a surface $S \in \MM_{g,n}$ with a systole $\gamma$ of length
$$
\ell(\gamma) < s_{g,n}:=\min\{2 \arcsinh \frac{1}{\sinh  \BB_{g,n}},2 \arcsinh 1\}
$$
so that any geodesic that crosses it has length at least $2 \BB_{g,n}$. Note that because $\ell(\gamma)< 2 \arcsinh 1$, all systoles of $S$, if there are several, are disjoint. Thus a shortest pants decomposition of $S$ necessarily contains all the systoles of $S$. By using the length expansion lemma explained above, one can increase the length of all the systoles at least up until $s_{g,n}$ to obtain a new surface $S'$ such that the lengths of all simple closed geodesics disjoint from the systoles increase (strictly). In particular $\BB(S') > \BB(S)$ and $\sys(S')=s_{g,n}$.  Thus we have moved to the thick part of moduli space (in this instance meaning the subspace of $\MM_{g,n}$ where $\sys(M) \geq s_{g,n}$ for all all $M\in \MM_{g,n}$) while increasing the Bers' constant. The thick part of moduli space being compact \cite{mu71}, it suffices to find the $\sup$ for $\BB_{g,n}$ on a compact set. Now $\BB$ is a continuous function over moduli space and this proves the existence of a $S_{\max}$.
\end{proof}

Note that in the following section we shall show that a stronger result holds for hyperbolic punctured spheres  (corollary \ref{cor:bigsys}): the constants $s_{0,n}$ defined above can be taken to be $2\, \arcsinh 1$.

\begin{definition} For a given surface $S$ and a given a pants decomposition $\PP$ such that $\ell(\PP) = \BB(S)$, a geodesic $\gamma \in \PP$ is said to be {\it essentially long} if there does not exist a multicurve $\mu \subset \PP \setminus \{\gamma\}$ and a multicurve $\mu'$ with $\mu\cup \mu'$ a pants decomposition of $S$ and $\ell(\mu') < \ell(\PP)$. In particular $\ell(\gamma)=\ell(P)$. 
\end{definition}

Note that, because $\ell(\PP) = \BB(S)$, the multicurve $\mu$ in the above definition is non-empty and $\ell(\mu) = \ell(\PP)$.

\begin{property}\label{property:maxbers2}
Let $S_{\max}\in \MM_{g,n}$ be such that $\BB(S_{\max})= \BB_{g,n}$. Then every simple closed geodesic on $S_{\max}$ intersects at least two distinct essentially long geodesics.
\end{property}

\begin{proof}
One begins, as in the previous property, by using the length expansion lemma to show that any simple closed geodesic $\gamma$ must intersect at least one essentially long curve. Geodesic length functions are convex along a twist \cite{ke83}, so one can twist  along $\gamma$ to make the essentially long curve even longer. As we have supposed that our surface is maximal, this means that another pants decomposition (of the twisted surface) must be at most as long as the previous one. The only lengths that change under a twist along $\gamma$ are those of geodesics that intersect $\gamma$. Thus $\gamma$ must have intersected a second essentially long geodesic. \end{proof}

Bers' constants satisfy the following inequalities.

\begin{property}\label{property:maxbers3}
The following inequalities hold:\\
\begin{enumerate}[a)]

 \item $ \BB_{g,n+1}>  \BB_{g,n}$, \\

\item $ \BB_{g,n}>  \BB_{g-1,n+2}$,\\

\item $ \BB_{g+1,n}>  \BB_{g,n}$.\\

\end{enumerate}
\end{property}

\begin{proof}
To show the first inequality, consider $S_{\max}\in \MM_{g,n}$ as in property \ref{property:maxbers1}. Now cut the surface open along some simple closed geodesic $\gamma$ (to obtain a new surface with two boundary geodesics $\gamma_1$ and $\gamma_2$). The idea is now to insert a pair of pants with two geodesics of the same length as $\gamma$ and a cusp. There is a two real parameter space of possible ways of doing this and one has to be careful about how the pasting is done. To do this, one chooses a point on $\gamma$, say $p$, and its lifts $p_1$ and $p_2$ on respectively $\gamma_1$ and $\gamma_2$. On the pair of pants set to be inserted, one considers the common geodesic perpendicular between the two boundary curves. Now one pastes the pair of pants so that one end of the common perpendicular coincides with $p_1$ and the other end coincides with $p_2$.\\

The construction we've described gives a surface $\tilde{S}$ of signature $(g,n+1)$ with an interesting property: any of its pants decompositions has length at least $\BB(S_{\max})$. To see this, consider the natural surjective map between simple closed curves on $\tilde{S}$ and simple closed curves on $S_{\max}$ one obtains by removing the extra cusp. (This map is generally called the {\it forgetful map}.) Now suppose that one has a shorter pants decomposition on $\tilde{S}$. The forgetful map applied to the curves in the pants decomposition gives a set of curves on $S_{\max}$ which have the same length if they did not intersect the inserted pair of pants, and which are strictly shorter if they did. This shows that $\BB_{g,n+1}\geq \BB_{g,n}$. But if $\BB(\tilde{S})=\BB(S_{\max})$, then neither $\gamma_1$ nor $\gamma_2$ intersect an essentially maximal geodesic: the length of any geodesic that crosses both $\gamma_1$ and $\gamma_2$ is strictly greater that the corresponding image by the forgetful map. So, by property \ref{property:maxbers2}, $\tilde{S}$ cannot be maximal for signature $(g,n+1)$. There must then be a surface with a larger Bers' constant of the same signature and this shows that $\BB_{g,n+1}> \BB_{g,n}$.\\

The second inequality is easier to show. The idea is similar to those that go into showing property \ref{property:maxbers1}  (a maximal surface in $\MM_{g,n}$ has a systole bounded below by the constant $s_{g,n}$). Now consider a maximal surface $S'$ for signature $(g-1,n+2)$ and transform it so that two cusps become genuine geodesics of length strictly less than $s_{g,n}$. By the Schwartz lemma principal, one can do this while increasing the lengths of all interior simple closed geodesics. In particular the length of all pants decompositions are strictly increased. Now one pastes the two boundary geodesics to obtain a surface $S''$ of signature $(g,n)$ where any short pants decomposition must contain the pasted geodesic because its length is less than $s_{g,n}$. And thus any pants decomposition of $S''$ is longer than a shortest pants decomposition of $S'$.\\

The third inequality is just an obvious application of the other two.
\end{proof}

\remark In particular one obtains that for closed surfaces $\BB_{g+1}>\BB_g$. The corresponding question for systoles on closed hyperbolic surfaces is open and seems to be a very difficult question. There is also a related question for (orientable) riemannian surfaces if whether the systolic ratio (supremum of the systole among Riemannian metrics of fixed area equal to $1$) is a decreasing function in genus and is also wide open. In particular, a positive solution to this question would show that all surfaces satisfy Loewner's inequality.\\

We shall use these properties to obtain upper bounds for punctured spheres using Buser's bounds (theorem \ref{thm:buser} above). Note that there are explicit bounds in \cite[Theorem 5.2.6]{bubook} for punctured spheres, but they are not as good.

\begin{corollary}\label{cor:basicpuncture}
Bers' constants for punctured spheres satisfy 
$$\BB_{0,n} \leq 3 \sqrt{3\pi} (n-1).$$
\end{corollary}
\begin{proof}
The bound from theorem \ref{thm:buser} with the second inequality from property \ref{property:maxbers3} shows $\BB_{0,2g}< 6  \sqrt{3\pi} (g-1)$. The adjustment is just to avoid a parity issue for $n$.
\end{proof}

\section{Main lemmas}\label{sec:mainlemmas}


In this section we prove lemmas which will be crucial in our proof of the main theorems. \\

We will use the following classical lemma (called Besicovitch's lemma
\cite{bes52}, see also \cite{al64} and \cite{besp42}):

\begin{lemma}\label{lem:alm}
Let $D$ be a riemannian disk, and let $\gamma=\partial D$. Suppose $\gamma$ is the concatenation of four subpaths, i.e., $\gamma= c_1 \cup c_2 \cup c_3 \cup c_4$. Then

$$
\area(D) \geq d_D(c_1,c_3) \,d_D(c_2,c_4).
$$
\end{lemma}

We shall use this lemma to prove the following result for hyperbolic spheres with cone point singularities. Note that a cone point of angle $0$ is in fact a cusp. This more general lemma will allow us to treat the case of hyperelliptic surfaces later, as well as the case of punctured spheres.

\begin{lemma}\label{lem:geocheeger}
Let $S$ be a hyperbolic sphere with $n\geq 5$ cone points of angle $\theta \in [0,\pi]$. 
Consider the set $\FF$ of simple closed non trivial geodesics $\delta$ such that each connected component of $S\setminus \delta$ contains at least $\frac{n}{4}$ cone points. Then the following inequality holds:
$$
\min_{\delta\in \FF} \ell(\delta) \leq 4 \sqrt{(2\pi-\theta) \, n_2-4\pi}
$$
where $n_1\leq n_2$ are the number of cone points lying in each connected component of $S\setminus \delta$.
\end{lemma}

\begin{remark} The bound on the length of $\delta$ roughly behaves like square root in the number of punctures. Due to the relationship between Cheeger constants, and the first eigenvalue of the laplacian, one could not hope for a similar lemma for closed surfaces with anything better than a linear bound in genus. Indeed, Brooks constructed surfaces \cite{br99} with first eigenvalue uniformly bounded from below which implies, via Cheeger's inequality for surfaces \cite{ch70}, that on such surfaces a curve $\delta$ is in the above lemma would necessarily have length at least linear in genus.
\end{remark}

\begin{proof}
Consider $\gamma\in \FF$ such that  $\ell(\gamma)=\min_{\delta\in \FF} \ell(\delta)$ ($\gamma$ exists and this is really a ``$\min$" as on a hyperbolic surface there are only a finite number of closed geodesics of length less than any given constant $K$).

Now consider the two connected components $S_1,S_2$ of $S\setminus \gamma$ with $n_1$ and $n_2$ cone points respectively ($n_1 \leq n_2$). We continue to denote $\gamma$ the resulting boundary geodesic on both $S_1$ and $S_2$ from cutting along $\gamma$. On $S_2$, consider any two distinct points $p,q$ on $\gamma$ and any geodesic path $c$ between them which is not a subpath of $\gamma$. Denote $\gamma'$ and $\gamma''$ the two subpaths of $\gamma$ separated by $p$ and $q$. Note that either the concatenation of $c$ with $\gamma'$ or the concatenation of $c$ with $\gamma''$, is a simple closed curve whose simple closed geodesic representative lies in $\FF$. Thus the following inequality holds:

$$
\ell(c) \geq \min\{\ell(\gamma'),\ell(\gamma'')\}.
$$

Now separate $\gamma$ into four arcs of length $\frac{\ell(\gamma)}{4}$, say $\gamma_k$, $k\in \Z_4$ (in cyclic ordering which follows a given orientation of $\gamma$). Because of the above inequality we have 

$$
d_{S_2}(\gamma_k,\gamma_{k+2}) \geq \frac{\ell(\gamma)}{4}
$$

\noindent for $k=1,2$. Now by lemma \ref{lem:alm}, we have that

$$
\area(S_2) \geq  \left(\frac{\ell(\gamma)}{4}\right)^2.
$$

As $S_2$ has one boundary component and $n_2$ cone points of angle $\theta$, we have $\area(S_2) = (2\pi-\theta) \, n_2-4\pi$ and one obtains the desired inequality. 
\end{proof}

The two cases we will be interested with in the sequel are the case $\theta=0$ (sphere with cusps) and the case $\theta=\pi$ (quotient sphere of a hyperelliptic surface by its hyperelliptic involution).\\

The second lemma we shall prove is essentially a topological lemma which we make geometric. The basic question here is: given a pants decomposition $\PP$, and a simple closed geodesic $\gamma$, how can one find a new pants decomposition which contains $\gamma$ and whose length is bounded by a function of the lengths of $\gamma$ and $\PP$?\\

The real problem is in actually constructing a new pants decomposition. Given $\gamma$ and $\gamma'$ two intersecting curves, there is a natural way of building a new curve $\gamma''$ which doesn't intersect $\gamma$ by concatenation of an arc in $\gamma$ and an arc in $\gamma'$ disjoint from $\gamma$. Roughly speaking, this is one of the ideas behind subsurface projections between curve complexes defined in \cite{masmin00}. This naturally gives a set of curves which are not necessarily pairwise disjoint but whose length are bounded by the sum of the lengths of $\gamma$ and $\gamma'$. When applying this technique to entire pants decompositions however, one has to ensure that within the set of projected curves there is a pants decomposition and in fact this works quite nicely on punctured spheres as we shall see.

\begin{lemma}\label{lem:proj}
Consider $S$ a hyperbolic sphere with boundary consisting of either cusps, cone points (of angle less than $\pi$) or boundary geodesics. Consider a pants decomposition $\PP$ of $S$ and a simple closed geodesic $\gamma$. Then there exists a pants decompostion $\PP'$ of $S$ containing $\gamma$ such that 
$$
\ell(\PP')\leq \ell(\PP)+\ell(\gamma).
$$
\end{lemma}

\begin{proof}
We begin by noting that all the steps in this proof are topological and
thus we can consider all boundary elements of our sphere as being
topological holes. This is possible because cone points of angle less than
$\pi$ ``behave" like cusps, or boundary geodesics \cite{padrcol}. The only
ambiguity would arise when one considers several cone points of angle
$\pi$, but as we will explain in the sequel (the remark following lemma
\ref{lem:bersde5}), we allow geodesics to be the concatenation of paths
between two cone points of angle $\pi$.

Also note that if $\gamma$ belongs to $\PP$ then there is nothing to prove
as one can take $\PP'$ to be $\PP$. Suppose now that $\PP$ and $\gamma$
intersect.

We shall construct the pants decomposition $\PP'$ with the desired
properties. We begin the construction of $\PP'$ by adding $\gamma$. Note
that the geodesic $\gamma$ separates the sphere into two connected
components, both spheres with boundary. The construction of $\PP'$ works
the same way on both ``hemispheres" (by hemispheres we mean the two
connected components of $S\setminus \gamma$).

We begin with the following observation. First consider a pair of pants
intersected by a curve $\delta$, possibly many times. It's not too
difficult to see that each connected component of the complementary region
on the pair of pants can only be one of the following:
\begin{enumerate}[i)]
\item\label{case1} a quadrilateral, where two arcs belong to $\delta$,
\item\label{case2} a one-holed bigon, where one of the arcs of the bigon
belongs to $\delta$, and the hole is one of the pant curves,
\item\label{case3} a one-holed quadrilateral where every second edge is an
arc of $\delta$,
\item\label{case4}  a hexagon where every second edge is an arc of $\delta$,
\item\label{case5}  an octagon where every second edge is an arc of
$\delta$ and among the remaining edges, two of them are distinct arcs of
the same boundary geodesic of the pair of pants,
\item\label{case6}  a pair of pants if the curve $\delta$ doesn't
intersect the pair of pants at all.
\end{enumerate}

As a consequence of this observation, it follows that $S\setminus \PP
\setminus \gamma$ is a collection of elements of the above list where the
curve $\delta$ above is now our chosen curve $\gamma$.\\

Consider one of the hemispheres, say $S'$. As by hypothesis, $\gamma$
intersects $\PP$, the part of $\PP$ contained in $S'$ consists of a
non-empty set $\arc$ of arcs with endpoints lying on $\gamma$, and a set
of curves entirely contained in $S'$. We start by adding to  $\PP'$ the
set of curves entirely contained in $S'$. The question is what to do with
the set of arcs. From the above case discussion, it follows that
$S'\setminus \arc$ consists in a union of elements of the above list.

Now consider the graph $T$  where vertices correspond to connected
components of $S'\setminus \arc$, and edges to arcs of $\arc$ : two
vertices share an edge if they have a boundary arc in common. Because each
arc separates $S'$, this graph is a tree. Furthermore, all vertices have
valency $2$, $3$ or $4$ (the valency of a vertex corresponds to the number
of arcs of the boundary of the considered connected component that
intersect $\gamma$).

This tree structure allows us to define an algorithm. One begins with some
edge $c_0$ and then proceeds step by step along the two oriented subtrees
$T_1$ and $T_2$ obtained by deleting this edge, where the orientation on
each subtree $T_i$ is the only one for which the remaining vertex of the
edge $c_0$ is the root. In the sequel, we will see that in order to ensure
the inequality we are proving, we may have to be careful with our choice
of initial arc. If one is not careful about this initial choice, one does
obtain a relatively short pants decomposition, but whose length is only
bounded by $ 2\ell(\gamma)+\ell(\PP)$. In order to choose our initial arc
in a satisfactory way, observe that each arc of $\arc$ cuts $\gamma$ in
two subarcs, then associate to each directed edge $c$ of $T\setminus
\{c_0\}=T_1\cup T_2$ the portion of $\gamma$ towards which we are going
and denote it by $\gamma_c$. It is possible to choose an initial edge
$c_0$ so that for any directed edge $c$  of $T\setminus \{c_0\}$,
$$
\ell(\gamma_c)\leq\ell(\gamma)/2.
$$
We choose such an edge as our initial edge.\\

\markleft{F. BALACHEFF AND H. PARLIER}

\noindent The initial algorithm step is the following: the arc $c_0$
separates $\gamma$ into two subarcs $\gamma_1$ and $\gamma_2$. We add to
$\PP'$ the two closed geodesics in the homotopy classes of  $\gamma_1\cup
c_0$ and $\gamma_2\cup c_0$. \\

\noindent Then the algorithm works vertex by vertex along the two subtrees
$T_1$ and $T_2$. Now suppose we are at a vertex $v$ of $T \setminus
\{c_0\}$. We consider the connected component $U$ of $S'\setminus \arc$
corresponding to $v$, and  the arc $c$ corresponding to the oriented edge
with endpoint $v$ (set $c=c_0$ if $v$ coincides with the root of $T_1$ or
$T_2$). We proceed as follows:

\begin{enumerate}[i)]
\item Suppose $U$ is of type \ref{case1}). This means the following arc
$c'$ is freely homotopic to $c$. So we don't add any curves to $\PP'$  and
go to the next connected component (there is only one). \\

\item Suppose $U$ is of type \ref{case2}).  We don't add any curves to
$\PP'$ and the process (in this direction) ends. \\

\item Suppose  $U$ is of type \ref{case3}). One has a one-holed
quadrilateral with one side being $c$ and the opposite side to $c$ is an
arc $c'$. We add to $\PP'$ the closed geodesic in the homotopy class of
$c' \cup \gamma_{c'}$. Observe that $\ell(c' \cup \gamma_{c'})<
\ell(\PP)+\ell(\gamma)$, and now proceed to the arc $c'$.\\

\item Suppose $U$ is of type \ref{case4}). Then one has a hexagon with $c$
and two other arcs $c_1$ and $c_2$ as three non-adjacent sides and the
remaining three sides are arcs of $\gamma$. We add to $\PP'$ the  two
closed geodesics in the homotopy class of $c_1 \cup \gamma_{c_1}$ and $c_2
\cup \gamma_{c_2}$. Observe that $\ell(c_i \cup \gamma_{c_i})<
\ell(\PP)+\ell(\gamma)$  for $i=1,2$, and now proceed to the other arcs
$c_1$ and $c_2$.\\

\item Suppose  $U$ is of type \ref{case5}). This is the most difficult
case to be treated and there are two sub-cases depending on the initial
arc $c$. Recall that on such an octagon, two of the sides are distinct
arcs of the same pants decomposition geodesic. The initial arc $c$ could
be one of these two arcs or not. Denote $\tilde{c}$ the opposite side to
$c$, and denote the two remaining arcs  $c_1$ and $c_2$. (The other four
sides belong to $\gamma$.) \\

\begin{enumerate}[a)]
\item Consider the case where $c$ and $\tilde{c}$ belong to the same pants
decomposition geodesic. We add to $\PP'$ the closed geodesics
corresponding to the curves $\tilde{c} \cup \gamma_{\tilde{c}}$, $c_1 \cup
\gamma_{c_1}$, $c_2 \cup \gamma_{c_2}$ and $c \cup (\gamma_c \setminus
\gamma_{\tilde{c}}) \cup \tilde{c}$. Notice that the length of the least
closed geodesic is bounded by $\ell(\gamma) + \ell(c)+ \ell(\tilde{c})$,
and as $c$ and $\tilde{c}$ belong to the same pants geodesic, we have
$\ell(c)+\ell(\tilde{c})\leq \ell(\PP)$.\\

\item Now consider the case where $c_1$ and $c_2$ belong to the same pants
decomposition geodesic. We begin by adding to $\PP'$ the closed geodesics
corresponding to the curves $\tilde{c} \cup \gamma_{\tilde{c}}$, $c_1 \cup
\gamma_{c_1}$ and $c_2 \cup \gamma_{c_2}$. Now orient $c_1$ and $c_2$ in
such a way that their initial point is closer to the subarc $c$ than their
final point. Then denote by $\alpha$ (respectively $\alpha'$) the subarc
of $\gamma$ going from the initial point of $c_1$ to the initial point of
$c_2$ (resp.  going from the final point of $c_1$ to the final point of
$c_2$). We add to $\PP'$ the closed geodesic in the homotopy class of
$\alpha \star c_2 \star (\alpha')^{-1}\star c_1 ^{-1}$ where $\star$
denote the concatenation.

If the portion $\alpha$  of $\gamma$ is no longer than $\ell(\gamma)/2$,
we have that the total length of the closed geodesic in the homotopy
class $\alpha \star c_2 \star (\alpha')^{-1}\star c_1 ^{-1}$ is less than
$2 \, \ell(\alpha) + \ell(c_1) + \ell(c_2)\leq \ell(\gamma)+ \ell(\PP)$
as $c_1$ and $c_2$ lie on the same pants geodesic. If not, $c$ is
necessarilly our initial arc $c_0$. In that case, we reinitiate the
algorithm from $c_1$. Note that except for possibly the initial arc
$c_1$, all outgoing arcs have associated lengths $\leq\ell(\gamma)/2$ (in
particular $c$ does because we've switched its direction). And our switch
is such that we are in the case v)a) explained above.\\
\end{enumerate}
\end{enumerate}

The curves constructed bound genuine pairs of pants. Furthermore they are
all disjoint. When the process ends, one obtains a full pants
decomposition $\PP'$ of $S'$. By construction, the length of each curve
$\delta$ of $\PP'$ that is not in $\gamma \cup (\PP\cap \PP')$  satisfies
$\ell(\delta')\leq \ell(\gamma)+\ell(\PP)$. We apply the same process to
the other hemisphere $S''$ pf $S \setminus \gamma$ and the result follows.
\end{proof}

\remark In the proof, every curve $\delta\in \PP$ that intersected $\gamma$ was replaced by a curve $\delta'$ whose length was roughly estimated as at most $\ell(\gamma) + \ell(c)$ where $c$ was an arc of $\delta$ from $\gamma$ to $\gamma$. If one considers a bound which uses  hyperbolic trigonometry one can obtain the following bound:

$$
\ell(\delta') <2 \arccosh (\sinh \frac{\ell(\gamma)}{2} \sinh\frac{\ell(c)}{2} ).
$$

In particular, if $\ell(\gamma) \leq 2\arcsinh 1$, then $\ell(\delta')< \ell(c)<\ell(\delta)$. Thus, as a corollary of the proof of lemma \ref{lem:proj}, one obtains the following result.

\begin{corollary}\label{cor:bigsys} Let $S_{n,\max}$ be a hyperbolic sphere with $n$ cusps such that it realizes the optimal Bers' constant, i.e., $S_{n,\max}$ is such that any of its pants decompositions contain at least one geodesic of length greater or equal to $\BB_{0,n}$. Then all simple closed geodesics of $S_{n,\max}$ have length strictly greater than $2\arcsinh 1$.
\end{corollary}
\begin{proof}
Suppose this was not the case. By property \ref{property:maxbers2}, a systole $\gamma$ must intersect $2$ an essentially long curve. Consider the associated pants decomposition $\PP$. Now by the above lemma, and in particular by using the bound from the remark above, one can obtain a new pants decomposition $\PP'$ with all geodesics $\delta'$ whose length satisfies
$$
\ell(\delta')<\ell(c)
$$
where $c$ is an arc of a curve from $\PP$. As all arcs from curves of $\PP$ have length at most the length of $\PP$, we obtain $\ell(\delta')<\ell(\PP)$ for all $\delta'$ and thus 
$$
\ell(\PP')<\ell(\PP).
$$
This is a contradiction as pants decomposition $\PP$ was supposed to be of minimal length.
\end{proof}

It might be interesting to know whether maximal surfaces in the closed case have the same property, by which we mean systole length bounded below by some universal constant.

 \section{Bers' constants for punctured spheres}\label{sec:cusps}
 
We are now set to prove the upper-bounds on Bers' constants for punctured spheres.

\begin{theorem}\label{th:berspuncturedsphere}
The following inequality holds for $n\geq 4$:

$$
\BB_{0,n} \leq 30 \sqrt{2\pi (n-2)}.
$$

\end{theorem}

\begin{remark}

In particular for any hyperbolic sphere $S$ with $n$ cusps we have

$$
\BB(S) \leq 30 \sqrt{\area(S)}.
$$

\end{remark}

\begin{proof}
We shall prove the theorem by induction. Note that by the bound from corollary \ref{cor:basicpuncture} the theorem is true for $n\leq 63$. \\

Now take $n\geq n_0=63$ and $S$ a hyperbolic sphere with $n$ cusps. By lemma \ref{lem:geocheeger}, one can find a ``short" geodesic $\gamma$ such that both connected components of $S\setminus \gamma$ contain at least $\frac{n}{4}$ cusps. Let $n_1\leq n_2$ with $n_1+n_2$ be the number of cusps of $S_1$ and $S_2$, the two connected components of the complementary region on $S$ to $\gamma$. The bound on the length of $\gamma$ from lemma \ref{lem:geocheeger} is precisely 

$$
\ell(\gamma) \leq 4\sqrt{2\pi (n_2-2)}
$$
as a cusp is a cone point of angle zero.\\

 Consider two surfaces $\tilde{S}_1$ and $\tilde{S}_2$ obtained by pasting a pair of pants (with two cusps and one simple closed geodesic with the same length as $\gamma$) onto $S_1$ and $S_2$. Note that we do not impose anything on the twist parameter of the gluing. The surfaces $\tilde{S}_1$ and $\tilde{S}_2$ are hyperbolic spheres with respectively $n_1+2$ and $n_2+2$ cusps.\\
 
The idea will be to use a short pants decomposition on $\tilde{S}_1$ and $\tilde{S}_2$ given by induction and then to use lemma \ref{lem:proj} to find short pants decomposition on these surfaces that contains $\gamma$, and thus one finds a short pants decomposition of both $S_1$ and $S_2$. From these pants decompositions, one finds the short pants decomposition of $S$ by pasting $S_1$ to $S_2$ with their pants decompositions to obtain $S$. \\


By induction, there is a pants decomposition $\PP_i$ of $\tilde{S}_i$ for $i=1,2$ with 
$$
\ell(\PP_i) \leq C \sqrt{2\pi n_2}
$$
with $C=30$. Using lemma \ref{lem:proj}, we can find a pants decomposition $\PP$ of $S$ containing $\gamma$ this time and such that 

\begin{eqnarray}
\nonumber \ell(\PP) &\leq& \max_{i=1,2}\ell(\PP_i) +\ell(\gamma)\\
\nonumber &\leq& C \sqrt{2\pi n_2} + 4\sqrt{2\pi(n_2-2)}\\
\nonumber &\leq& (C+4) \sqrt{2\pi n_2}\\
\nonumber &\leq& (C+4) \sqrt{{3 \over 4}2\pi n}\\
\nonumber
\end{eqnarray}
as $n_2\leq {3 \over 4} n$.
Now observe that $C$ is the smallest integer number such that 
$$
(C+4) \sqrt{{3 \over 4} n}\leq C \sqrt{n-2}
$$
for any $n\geq n_0$. Thus the theorem is proved.
\end{proof}

\begin{remark}
In the proof, we have chosen $n_0$ such that the constant $C$ is small as possible. 
\end{remark}

 \section{Bers' constants for hyperbolic spheres with cone points of
   angle $\pi$}\label{sec:cone}
 
 In this section, we prove the analog of Theorem \ref{th:berspuncturedsphere} for hyperbolic spheres with cone points of angle $\pi$. This result will be used in the next section in order to obtain the upper-bound on Bers' constants of hyperelliptic surfaces (see theorem \ref{th:bershypsurf}). \\ 

First we will need the following result in order to initiate the induction.
  
\begin{lemma}\label{lem:bersde5}
Let $S$ be a hyperbolic sphere with $n \geq 5$ cone points of angle $\pi$. 
Then there exists a pants decomposition $\gamma_1,\ldots,\gamma_{n-3}$ satisfying
$$
\ell(\gamma_k)\leq  4k \, \ln {4\pi(n-2) \over k}
$$
for $k=1,\ldots,n-3$.
\end{lemma}

\begin{remark}
Here we authorize a closed geodesic to be the concatenation with its inverse of a geodesic arc from a cone singularity to another. As the cone angles are equal to $\pi$, this produces a simple closed geodesic of length twice the length of the arc between the two cone points.
\end{remark}

\begin{proof}
We proceed as in the proof of \cite[Theorem 5.2.3]{bubook}. 
Let $\{p_1,\ldots,p_n\}$ denote the cone points of angle $\pi$. First we cut $S$ along the closed geodesics $\gamma_1,\ldots, \gamma_m$ of length $\leq 2 \arcsinh 1$ if there are any. By the collar theorem these geodesics are pairewise disjoint and $m\leq n-3$. If there are no such geodesics on $S$, we consider the (closed) ball of radius $r$ centered at $p_1$. For small $r>0$ this ball is embedded and contains one cone point (its center). We increase $r$ until something happens: either the
ball stops to be embedded, or meets another cone point. We denote by $r_1$ the corresponding radius. The area of $B(p_1,r_1)$ is equal to
$\pi (\cosh(r_1)-1)$ and is less than the area of the surface $\area(S)= \pi(n-4)$, so $r_1 \leq \arccosh\,{2(n-2)}$. \\

\noindent {\it First case.} If the ball meets another cone point, say $p_2$ up to reordering the cone points, we get the first closed geodesic denoted by $\gamma$ made of the concatenation with its inverse of a geodesic arc from $p_1$ to $p_2$. \\

\noindent {\it Second case}.  If the ball admits a point of auto-intersection, there exists a geodesic loop $\tilde{\gamma}$ based at $p_1$ of length $2r_1$ as it can be easily seen by passing to the universal covering. We decompose $S \setminus \tilde{\gamma}=\tilde{S}_1\cup \tilde{S}_2$ into two closed spheres with boundary component $\tilde{\gamma}$ and cone points such that the numbers of their cone singularities $n_1$ and $n_2$ satisfy $n_1+n_2-1$ and $n_2 \geq n_1 \geq 2$. As the boundary of both $\tilde{S}_1$ and $\tilde{S}_2$ is convex, there exists a closed geodesic $\gamma_i$ in each $\tilde{S}_i$ corresponding to the homotopy class of $\tilde{\gamma}$ whose length is less than $2 r_1$. \\

Now we proceed by induction as in \cite[Theorem 5.2.3]{bubook}. Assume that after many such steps, the pairewise disjoint closed geodesics $\gamma_1,\ldots,\gamma_k$ have been found, and let $S^k$ be a disjoint reunion of sphere with one boundary component which remains after cutting $S$ open along $\gamma_1,\ldots,\gamma_k$ and removing the connected components which are pair of pants or a cylinder with a cone singularity in its interior. Assume by induction that 
$$
\ell(\partial S^k)\leq  4k \, \ln {4\pi(n-2) \over k}
$$
and that 
$$
\ell(\gamma_j)\leq  4j \, \ln {4\pi(n-2) \over j}
$$
for $j=1,\ldots,k$.

For $r>0$, we define $Z(r)$ to be the points of $S^k$ at distance less than $r$ from $\partial S^k$. For $r$ small enough, $Z(r)$ is a disjoint union of half-collars of area
$$
\area ( Z(r) )=\ell (\partial S^k) \sinh r.
$$
When $r$ grows, two types of catastrophes can happen: the first one is that one of the half-collars ceases to be embedded. That case corresponds to the second case treated in \cite[Theorem 5.2.3]{bubook}. The second possible catastrophe is that one of the half-collars meets a cone point singularity. In that case, we can consider the tubular neighborhood of the union of the boundary curve corresponding to the half-collar and of a geodesic arc going from the boundary to the cone point. The boundary of this tubular neighborhood is a non-contractible curve and we define $\gamma_{k+1}$ to be the unique closed geodesic in this homotopy class. So we cut $S^k$ open along $\gamma_{k+1}$ and define $S^{k+1}$ in the obvious way. We can argue as in  \cite[Theorem 5.2.3]{bubook} to prove that 
$$
\ell(S^{k+1})\leq  4(k+1) \, \ln {4\pi(n-2) \over (k+1)}.
$$
\end{proof}

\begin{remark}
We then have
$$
\BB_{\pi,n} \leq 4 \ln 6 \pi \, (n-3).
$$
\end{remark}

Using this, we can now apply our method to obtain the following upper-bound on Bers' constants of spheres with cone points of angle $\pi$.

\begin{theorem}\label{th:berssphereconic}
The following inequality holds for $n\geq 5$:

$$
\BB_{\pi,n} \leq 10 \ln 6\pi \,  \sqrt{n-4} \left(<17 \sqrt{\pi(n-4)}\, \right).
$$

\end{theorem}

\begin{remark}
We can obtain similar results for hyperbolic spheres with cone singularities of angles
$\theta_1, \ldots,\theta_n$ between $0$ and $\pi$. The constant involved in
the inequality will depend on the different values of the
cone angles $\theta_1,\ldots,\theta_n$ and will be uniformly bounded from above by some universal constant which comes from the extremal case where all angles are equal to $0$, i.e., there are $n$ cusps.
\end{remark}

\begin{proof}
By lemma \ref{lem:bersde5} the result is true for $5\leq n\leq 8$.
Now take $n\geq 9$ and $S$ a hyperbolic sphere with $n$ cone points of
angle $\pi$. By lemma \ref{lem:geocheeger}, one can find a short geodesic $\gamma$ such that both connected components of $S\setminus \gamma$ contain at least $\frac{n}{4}$ cone points. Let $n_1\leq n_2$ with $n_1+n_2$ be the number of cone points of $S_1$ and $S_2$, the two connected components of the complementary region on $S$ to $\gamma$. Observe that necessarily both $S_1$ and $S_2$ contains at least three cone points. The bound on the length of $\gamma$ from lemma \ref{lem:geocheeger} is precisely 

$$
\ell(\gamma) \leq 4\sqrt{\pi (n_2-2)}.
$$

Consider two surfaces $\tilde{S}_1$ and $\tilde{S}_2$ obtained as follows from
$S_1$ and $S_2$ : fix a pair of opposite points on $\gamma$ and identify the
two geodesic subarcs of $\gamma$ defined by this pair of points such that the resulting surfaces $\tilde{S}_1$ and $\tilde{S}_2$ are hyperbolic spheres with respectively $n_1+2$ and $n_2+2$ cone points of angle $\pi$.\\
 
By induction, there is a pants decomposition $\PP_2$ of $\tilde{S}_2$ with 

$$
\ell(\PP_2) \leq C \sqrt{\pi(n_2-2)} 
$$
where $C$ denote for simplicity the value $10 \ln 6\pi$.
Here we authorize a closed geodesic to be the concatenation with its inverse of a geodesic arc from a cone singularity to another. 
 Using lemma \ref{lem:proj}, we can find a pants decomposition $\PP_2'$ containing $\gamma$ this time and such that 

$$
\ell(\PP_2') \leq \ell(\PP_2) +\ell(\gamma).
$$
Thus

\begin{eqnarray}
\nonumber \ell(\PP_2') &\leq&(C+4)\sqrt{\pi (n_2-2)}\\
\nonumber &\leq&(C+4)\sqrt{\pi ({3n \over 4}-2)}\\
\nonumber &\leq&{\sqrt{3} \over 2}(C+4)\sqrt{\pi(n-{8\over 3})}\\
\nonumber
\end{eqnarray}

This above quantity is indeed smaller than 

$$
C \sqrt{\pi(n-4)}
$$
if for any $n\geq 9$ 
$$
C\geq f(n):={4\sqrt{3n -8} \over 2\sqrt{n-4}-\sqrt{3n-8}}.
$$
But the function $f$ is easily checked to be decreasing and as $f(9)<11$ this ends the proof.

\end{proof}

\section{Bers' constants for hyperelliptic surfaces}\label{sec:hyper}

Using the result of the preceding section, we shall prove roughly asymptotically optimal upperbounds on Bers' constants for hyperelliptic surfaces. The general strategy will be to look at the sphere quotient of the hyperelliptic surface, find a short pants decomposition by the results of the preceding section and lift the pants decomposition by the hyperelliptic involution. The lift of a pants decomposition, although not a pants decomposition, lifts to a multicurve whose compementary region is a collection of 3 and 4-holed spheres. This is the object of the first lemma.

\begin{lemma}
We consider a hyperelliptic surface $\tilde{S}$ and its quotient sphere $S= \tilde{S}/\sigma$ by the hyperellipic involution $\sigma$. A pants decomposition $\PP=\{\gamma_1,\ldots,\gamma_{2g-1}\}$ of $S$ lifts to a multicurve $\mu$ on $S$ with complementary region a collection of $3$ or $4$-holed spheres and their lengths satisfy
$$
\ell(\mu)\leq 2 \, \ell(\PP).
$$ 
\end{lemma}

\begin{proof}
First note that the lift of a closed geodesic $\gamma$ on $S$ by $\sigma$ is 
\begin{itemize}
\item either a closed geodesic $\tilde{\gamma}$ of the same length, {\it i.e.}  $\ell(\gamma)=\ell(\tilde{\gamma})$,  if $\gamma$ consists of  the concatenation with its inverse of a geodesic arc from a cone singularity to another,
\item either a closed geodesic $\tilde{\gamma}$ with $\ell(\gamma)=2 \, \ell(\tilde{\gamma})$, if $\gamma$ bounds an odd number of cone points, 
\item or a pair of disjoint closed geodesic $\tilde{\gamma}_1$ and $\tilde{\gamma}_2$ with $\ell(\gamma)=\ell(\tilde{\gamma}_i)$ for $i=1,2$, if $\gamma$ bounds an even number $n\geq 4$ of cone points.
\end{itemize}
Here we say that $\gamma$ bounds an even (respectively odd) number of cone points if each connected  component of $S\setminus \gamma$ contains an even (respectively odd) number of cone points in its interior. 

So to prove our assertion we will show that the lift of a connected component of $S\setminus \PP$ consists either of a single $3$ or $4$-holed sphere of $\Sigma$ which is invariant by the hyperelliptic involution $\II$, or a pair of (interior disjoints) $3$-holed of $\Sigma$ which are image one from the other by $\II$. 

Consider a connected component $P$ of $S\setminus \PP$.  Either $P$ is the interior of a pair of pants bounded by three closed geodesics of $\{\gamma_1,\ldots,\gamma_{2g-1}\}$ , or $P$ is the interior of a cylinder bounded by two such closed geodesics with one cone singularity in its interior. In the first case, either each of the boundary geodesics  bounds an even number of cone singularities and then $P$ lift to a pair of disjoint $3$-holed sphere $\tilde{P}_1$ and $\tilde{P}_2$ such that $\II(\tilde{P}_1)=\tilde{P}_2$, or two of the boundary geodesics bounds an odd number of cone singularities and then $P$ lift to a $4$-holed sphere invariant by $\II$. In the case where  $P$ is the interior of a cylinder bounded by two closed geodesics  with one cone singularity in its interior, $P$ lift to a $3$-holed sphere invariant by $\II$.
\end{proof}

As the lift of a pants decomposition may only decompose parts of the surface into 4-holed spheres, we now need the following lemma.

\begin{lemma}
A 4-holed sphere with boundary curves of length at most $\ell$ contains a simple closed geodesic of length at most $2  \ell +12$.
\end{lemma}

\begin{proof}
Consider a four holed sphere with all boundary lengths less than $\ell$. As mentioned before, by increasing the length of all smaller boundary curves, one can increase the length of all interior curves, so it suffices to consider the case where all four boundary curves are of length $\ell$. Similarly, if $\ell$ is less than $\frac{\pi}{\sinh 1}$ then by simultaneously increasing all four lengths, one also can increase the length of all interior geodesics. So we can now suppose that we have a four holed sphere with all boundary geodesics of length $\ell\geq \frac{\pi}{\sinh 1}$. Consider a shortest path $c$ between any two boundary curves. We can bound the length of $c$ as follows: consider an $r$ neighborhood of the four boundary curves. As long as this neighborhood is embedded, it has area $4\ell \sinh(r)$. Thus 
$$
4 \ell \sinh r \leq 4 \pi
$$
as the full area of the 4-holed sphere is $4\pi$ and thus
$$
r\leq \arcsinh\frac{\pi}{\ell}.
$$
But as $\ell\geq \frac{\pi}{\sinh 1}$ we obtain $r\leq 1$. The geodesic in the homotopy class of the curve obtained by considering the $\varepsilon$ boundary of the two boundary geodesics and $c$ has length less than $2\ell + 4$. Now the inequality of the lemma holds for $\ell<\frac{\pi}{\sinh 1}$ the above argument show that the length is bounded by $2\pi/\sinh 1 +4 <12$.

\end{proof}
 
\begin{theorem}\label{th:bershypsurf}
Let $g\geq 2$.  Then
$$
\BB^{hyp}_g  < 40 \ln 6\pi \sqrt{2(g-1)} + 12\left(< 51\sqrt{4\pi(g-1)}\, \right).
$$
\end{theorem}

\begin{proof}
The quotient of a hyperelliptic surface $\tilde{S}$ by its hyperelliptic involution is a hyperbolic sphere $S$ with $2g+2$ cone points of angle $\pi$. 
By theorem \ref{th:berssphereconic} there exists a pants decomposition $\PP=\{\gamma_1,\ldots,\gamma_{2g-1}\}$ of $S$ by $2g-1$ closed geodesics of length bounded by $10 \ln 6\pi \,  \sqrt{2(g-1)}$. Now by the first of the two previous lemma, by the hyperellitpic involution this lifts to a multicurve $\mu$ on $\tilde{S}$ of length at most $20 \ln 6\pi \,  \sqrt{2(g-1)}$. $S\setminus \mu$ consists of a set of 3 and 4-holed spheres, and on each of the 4-holed spheres, by the previous lemma one can find a short geodesic to complete $\mu$ into a pants decomposition of $\tilde{S}$ which satisfies the above bound.
\end{proof}

\begin{remark} Using the same methods as in the hyperelliptic case, one finds a similar bounds for Bers' constants on so-called $M$-maximal surfaces. $M$-maximal surfaces are closed genus $g$ surfaces admitting an orientation reversing involution with the fixed point set consisting of $g+1$ disjoint simple closed geodesics. The quotient of such a surface by its involution is a sphere with $g+1$ boundary geodesics, and so it's not too difficult to see how to use the same method as for hyperelliptic surfaces.
\end{remark}

 \section{Bers' constants for hyperbolic spheres with boundary}\label{sec:boundary}
 
In this section we derive from our approach an upperbound on Bers' constants for hyperbolic spheres with geodesic boundary. \\

We begin by finding a bound that follows directly from theorem \ref{th:berspuncturedsphere} and lemma \ref{lem:proj}. This lemma will essentially be used to start up the induction.

\begin{lemma}\label{lem:roughboundary}
Let $S$ be a hyperbolic sphere with $n\geq 4$ geodesic boundary components of length at most $\ell$. Then the length of a shortest pants decomposition of $S$ satisfies the following upper bound.

$$
\BB(S) \leq 60\sqrt{\pi (n-1)} + n \ell.
$$
\end{lemma}

\begin{proof}
As explained previously, we can assume that all the boundary geodesics are of length $\ell$. We begin by gluing a pair of pants with two cusps and a geodesic of length $\ell$ to each boundary geodesic. Denote $\tilde{S}$ the sphere with $2n$ cusps thus obtained. By theorem \ref{th:berspuncturedsphere}, $\tilde{S}$ has a pants decomposition whose is at most
$$
30\sqrt{2\pi (2n-2)}.
$$

Now by lemma \ref{lem:proj}, we can find a new pants decomposition containing any given geodesic whose length is bounded by the sum of the original pants decomposition and the length of the curve. Applying this lemma iteratively with the original boundary geodesics one obtains the desired bound.
\end{proof}

\begin{lemma}\label{lem:cutspherewithboundary}
Let $S$ be a hyperbolic sphere with $n\geq 4$ (connected) boundary components of length at most $\ell$. 
Consider the set $\FF'$ of simple closed closed geodesics $\delta$ such that each connected component of $S\setminus \delta$ contains at least $\frac{n}{4}$ boundary components. Then the following inequality holds:
$$
\min_{\delta\in \FF} \ell(\delta) \leq 4 \sqrt{2\pi (n_2-2)+n_2{\ell^2 \over 2\pi}}
$$
where $n_1\leq n_2$ are the number of boundary components lying in each connected component of $S\setminus \delta$.
\end{lemma}

\begin{proof}
We proceed as in the punctured case. Consider $\gamma\in \FF'$ such that  $\ell(\gamma)=\min_{\delta\in \FF} \ell(\delta)$.
Now consider the two connected components $S_1,S_2$ of $S\setminus \gamma$ with $n_1$ and $n_2$ boundary components respectively. We still denote $\gamma$ the resulting boundary geodesic on both $S_1$ and $S_2$ from cutting along $\gamma$. On $S_2$, consider any two distinct points $p,q$ on $\gamma$ and consider any geodesic path $c$ between them which is not a subpath of $\gamma$. Denote $\gamma'$ and $\gamma''$ the two subpaths of $\gamma$ separated by $p$ and $q$. Note that either the concatenation of $c$ with $\gamma'$ or the concatenation of $c$ with $\gamma''$, is a simple closed curve whose simple closed geodesic representative lies in $\FF'$. Thus the following inequality holds:

$$
\ell(c) \geq \min\{\ell(\gamma'),\ell(\gamma'')\}.
$$

Now separate $\gamma$ into four arcs of length $\frac{\ell(\gamma)}{4}$, say $\gamma_k$, $k\in \{1,2,3,4\}$ (in cyclic ordering which follows a given orientation of $\gamma$). Because of the above inequality we have 

$$
d_{S_2}(\gamma_k,\gamma_{k+2}) \geq \frac{\ell(\gamma)}{4}
$$

\noindent for $k=1,2$. Now take $S_2$ and glue along each connected component $\eta$ of the boundary (except $\gamma$) a round hemisphere of radius $\ell(\eta) \over {2\pi}$. We denote by $\tilde{S}_2$ the disk thus obtained. As $\ell(\eta)\leq \ell$, we can easily bound its area by 
$$
\area(\tilde{S}_2)\leq\area(S_2)+n_2\, {\ell^2 \over 2\pi}.
$$
By lemma \ref{lem:alm}, we have that

$$
\area(\tilde{S}_2) \geq  \left(\frac{\ell(\gamma)}{4}\right)^2.
$$

As $S_2$ has $n_2+1$ geodesic boundary components, we have $\area(S_2) = 2\pi (n_2-1)$ and one obtains our inequality. \end{proof}

\begin{theorem}\label{thm:boundarysphere}
The following inequality holds:
$$
\BB_{0,n, \ell} \leq (30\sqrt{2}+2\sqrt{\pi}) \, \sqrt{2\pi(n-2)}\sqrt{\left({\ell \over 2\pi}\right)^2+1} < 46\, \sqrt{2\pi(n-2)}\sqrt{\left({\ell \over 2\pi}\right)^2+1}.
$$
\end{theorem}

\begin{proof}

We shall prove the theorem by induction. Note that by lemma \ref{lem:roughboundary}  the theorem is true for $n = 4$. \\

Now take $n\geq 5$ and $S$ a hyperbolic sphere with $n$ boundary components of length at most $\ell$. By lemma \ref{lem:cutspherewithboundary}, one can find a ``short" geodesic $\gamma$ such that both connected components of $S\setminus \gamma$ contain at least $\frac{n}{4}$ boundary components. Let $n_1\leq n_2$ with $n_1+n_2$ be the number of boundary components different from $\gamma$ of $S_1$ and $S_2$, the two connected components of the complementary region on $S$ to $\gamma$. The bound on the length of $\gamma$ from lemma \ref{lem:cutspherewithboundary} is precisely 

$$
\ell(\gamma) \leq 4 \sqrt{2\pi (n_2-2)+n_2{\ell^2 \over 2\pi}}.
$$

 Consider two surfaces $\tilde{S}_1$ and $\tilde{S}_2$ obtained by pasting a pair of pants (with two cusps and one simple closed geodesic with the same length as $\gamma$) onto $S_1$ and $S_2$. Note that we do not impose anything on the twist parameter of the gluing. The surfaces $\tilde{S}_1$ and $\tilde{S}_2$ are hyperbolic spheres with respectively $n_1+2$ and $n_2+2$ boundary components of length at most $\ell$.\\
 
The idea is the same as used in Theorem \ref{th:berspuncturedsphere}.
By induction, there is a pants decomposition $\PP_2$ of $\tilde{S}_2$ with 
$$
\ell(\PP_2) \leq C \, \sqrt{2\pi(n_2-2)}\sqrt{\left({\ell \over 2\pi}\right)^2+1}.
$$
In order to simplify the equations in the sequel, we denote $C$ the constant $30\sqrt{2}+2\sqrt{\pi}$. Using lemma \ref{lem:proj}, we can find a pants decomposition $\PP_2'$ containing $\gamma$ this time and such that 

$$
\ell(\PP_2') \leq \ell(\PP_2) +\ell(\gamma)
$$

thus

\begin{eqnarray}
\nonumber \ell(\PP_2') & \leq & C \, \sqrt{2\pi(n_2-2)}\sqrt{\left({\ell \over 2\pi}\right)^2+1} + 4 \sqrt{2\pi (n_2-2)+n_2{\ell^2 \over 2\pi}}\\
\nonumber &\leq& C \, \sqrt{2\pi(n_2-2)}\sqrt{\left({\ell \over 2\pi}\right)^2+1} + 4 \sqrt{2\pi (n_2-2)}\sqrt{1+{n_2 \over n_2-2}\left({\ell^2 \over 2\pi}\right)^2}\\
\nonumber &\leq& (C+4\sqrt{n_2 \over n_2-2}) \, \sqrt{2\pi(n_2-2)}\sqrt{\left({\ell \over 2\pi}\right)^2+1}\\
\nonumber &\leq& (C+4\sqrt{3}) \, \sqrt{2\pi(n_2-2)}\sqrt{\left({\ell \over 2\pi}\right)^2+1}\\
\nonumber &\leq&{\sqrt{3} \over 2}(C+4\sqrt{3})\sqrt{2\pi(n-{8\over 3})}\sqrt{\left({\ell \over 2\pi}\right)^2+1}.\\
\nonumber
\end{eqnarray}

This above quantity is indeed smaller than 

$$
C \, \sqrt{2\pi(n-2)}\sqrt{\left({\ell \over 2\pi}\right)^2+1}.
$$
if for any $n\geq 5$ 
$$
C\geq f(n):={4\sqrt{3}\sqrt{3n -8} \over 2\sqrt{n-2}-\sqrt{3n-8}}.
$$
But the function $f$ is easily checked to be decreasing and as $f(5)<C$ this ends the proof.

\end{proof}

\section{The hairy sphere and other examples} \label{sec:ex}

In this section we show how to construct a hyperbolic sphere with boundary which will be used to show that our order of growth of Bers' constants for punctured spheres is optimal and, as a byproduct, we get a better lower bound on Bers' constants for closed surfaces than the one given by the so-called ``hairy torus" example given in \cite{bubook}.\\

The basic construction is similar to the one for the hairy torus examples but differs in one essential way. The fact that the hairy torus examples have large Bers' constants is immediate once one remarks that, because a hairy torus has genus $1$, at least one curve in a pants decomposition must cut genus. Then one calculates the minimum length of a genus cutting curve and the lower bound follows. In the case of a punctured sphere (or a sphere with boundary components), there is no genus to cut. The following topological lemma is the trick to replace the genus cutting argument.

\begin{lemma}\label{lem:top} Let $S$ be a (topological) sphere with $n>5$ punctures. Let $\alpha$ and $\beta$ be disjoint simple essential closed curves such that each bounds a pair of pants. Then a pants decomposition of $S$ contains a curve $\gamma$ such that one of the following hold:
\begin{enumerate}[i)]
\item\label{lem:1} $\gamma$ is either equal to $\alpha$ or $\beta$,\\
\item\label{lem:2}  $\gamma$ crosses both $\alpha$ and $\beta$,\\
\item\label{lem:3}  $\gamma$ separates $\alpha$ from $\beta$.
\end{enumerate}
\end{lemma}

\begin{proof}
Let $\mathcal P$ be a pants decomposition of $S$. Suppose that neither (\ref{lem:1}) nor (\ref{lem:2}) hold for $\mathcal P$. Consider the set of curves $\gamma_1,\hdots,\gamma_m\subset {\mathcal P}$ that cross $\alpha$. Note that this set is non empty otherwise $\alpha\subset {\mathcal P}$. Denote $S_\alpha$ the connected component of $S\setminus \alpha$ containing $\beta$. Further consider the arcs, say $c_1,\hdots, c_l$ of $\gamma_1,\hdots,\gamma_k$ on $S_\alpha$.\\

Because every essential curve on a sphere is separating, one can organize the arcs into groups as follows. Note that any arc is separating. For two distinct arcs, say $c$ and $\tilde{c}$, we say that $c$ {\it includes} $\tilde{c}$ if $c$ separates $\tilde{c}$ from $\beta$. Now note that if $c$ includes $\tilde{c}$ and $\tilde{c}$ includes $\hat{c}$ then $c$ includes $\hat{c}$. Denote the set of arcs which are not included by any other arcs by $c_1,\hdots,c_k$ and such that, for a given orientation of $\alpha$, the endpoints of $c_i$ follow those of $c_{i+1}$ for all $i$ from $1$ to $k-1$. Note that it is a priori possible that $k=l$ but in any case $k>0$.\\

Now consider the curve $\gamma$ homotopic to the closed path constucted as follows. Suppose we've given $\alpha$ the orientation used above, and an orientation to the paths $c_i$ which coincides. The path is given by the concatenation of $c_1$ to the path on $\alpha$ between the endpoint of $c_1$ and the initial point of $c_2$, then the path $c_2$, and so forth until the path on $\alpha$ between the endpoint of $c_k$ and the initial point of $c_1$. The homotopy class $\gamma$ thus constructed is clearly essential and disjoint from  both $\alpha$ and the curves of $\mathcal P$. Thus is must be contained in $\mathcal P$. 
Finally it separates $\alpha$ from $\beta$ and as such it satisfies the condition (\ref{lem:3}). This proves the lemma.
\end{proof}

We can now construct the example of a hairy sphere and the previous lemma will ensure that any pants decomposition of this surface contains at least one long curve.\\

The basic building block is the same one as for the hairy torus, i.e., a hyperbolic cylinder with one short boundary geodesic and with the other boundary curve a piecewise geodesic as follows. Take a hyperbolic right-angled pentagon with two consecutive sides of equal length and such that the only side which does not touch these two sides has length $\ell/4$ with $\ell\geq 0$ a free parameter. Using the trigonometric formula for the right angled pentagon, one obtains that the two sides of equal length have length 
$$
x_0=2 \arcsinh\sqrt{\cosh \frac{\ell}{8}}.
$$
Paste four isometric copies of these to obtain a cylinder with one boundary curve a geodesic of length $\ell$ and the other is a right-angled quadrilateral with side length $2x_0$. For any integer $p$, as in the hairy torus example, paste together $p^2$ of these cylinders to form a ``hairy" square. Now take two copies of one of these squares and paste them together along an equal side to obtain a ``hairy rectangle" of side lengths $4p x_0$ and $2p x_0$. Next, paste together the two sides of the rectangle of length $2p x_0$ to obtain a ``hairy cylinder". The long boundary geodesics of the cylinder we shall denote $\alpha$ and $\beta$ for future use. Now take two pairs of pants of boundary lengths $\ell,\ell, 4p x_0$ and paste them (the twist parameter doesn't matter) to $\alpha$ and $\beta$ along the two long sides to obtain a hairy sphere with $2p^2+4$ boundary geodesics of length $\ell$. Denote this sphere $S_{p,\ell}$. We can now show the following.

\begin{proposition}\label{prop:hairy}
Any pants decomposition of $S_{p,\ell}$ has a curve of length at least $4p x_0$.
\end{proposition}

\begin{proof}
The result essentially follows from lemma \ref{lem:top}. Fix a pants decomposition $\mathcal P$ of $S_{p,\ell}$. The lemma tells us that $\mathcal P$ contains a curve $\gamma$ that is either equal to $\alpha$ or $\beta$ (both of length $4p x_0$), crosses both $\alpha$ and $\beta$ or separates $\alpha$ from $\beta$.\\

Suppose $\gamma$ crosses both $\alpha$ and $\beta$. Then $\gamma$ contains at least two geodesic distinct sub-paths that go from $\alpha$ to $\beta$ (because both $\alpha$ and $\beta$ are separating curves). Now any geodesic path from $\alpha$ to $\beta$ has length at least $2p x_0$.\\

\begin{figure}[h]
\leavevmode \SetLabels
\L (0.47*.71) $\alpha$\\
\L (.73*.45) $\beta$\\
\L (0.03*.45) $\alpha$\\
\L (.32*.45) $\beta$\\
\endSetLabels
\par
\begin{center}
\AffixLabels{\centerline{\epsfig{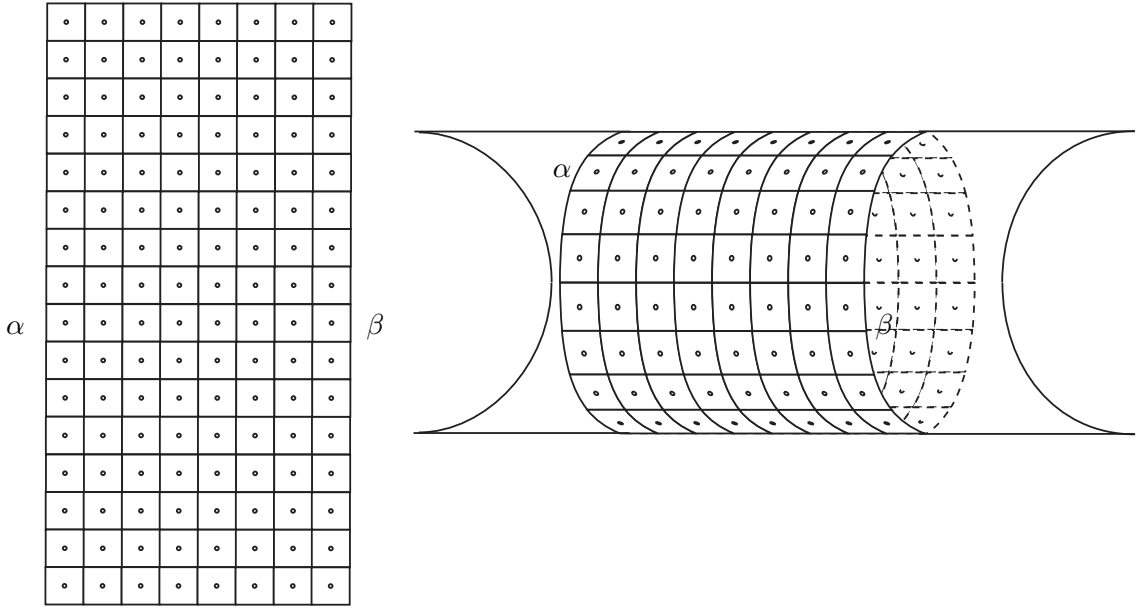}}}
\end{center}
\caption{The hairy rectangle and the hairy sphere}
\label{fig:hairy}
\end{figure}

Finally suppose $\gamma$ is a curve disjoint from both $\alpha$ and $\beta$ that separates $\alpha$ and $\beta$. Consider the ``horizontal" geodesic lines of figure \ref{fig:hairy} which become geodesic segments on $S_{p,\ell}$ which join $\alpha$ to $\beta$. The curve $\gamma$ must cross all of these lines otherwise it does not separate $\alpha$ from $\beta$. Now the minimum distance between any two successive horizontal lines is $2 x_0$. As $\gamma$ passes through two successive horizontal lines at least $2p$ times, its length is at least $4p x_0$. This proves the proposition.\end{proof}

As an immediate corollary to the above, we obtain lower bounds on Bers' constants for punctured spheres.

\begin{corollary}\label{cor:lowerNspheres}
For  $n>4$, Bers' constants $\BB_{0,n}$ satisfy the following inequality:
$$
\BB_{0,n} > 8 \arcsinh 1 (\sqrt{\frac{n-4}{2}}-1).
$$
\end{corollary}

\begin{proof}
One begins by writing $n= 2p^2+4+k$ where $k$ is the smallest non-negative integer satisfying this equality, i.e., $n\geq 2p^2+4$ and $n < 2(p+1)^2+4$. Thus we have

$$
p > \sqrt{\frac{n-4}{2}}-1.
$$

If $k=0$, then we construct the surface $S_{p,\ell=0}$ as above and we obtain the (slightly stronger) inequality:

$$
\BB_{0,n}\geq  8 p \arcsinh 1 = 8  \arcsinh 1  \sqrt{\frac{n-4}{2}}.
$$

For $k>0$ we take the surface $S_{p,\ell=0}$ as above except for one of the two pairs of pants which we attached to the ``hairy cylinder". On one of the pair of pants, say the one attached to $\alpha$, we choose one of the cusps to be a genuine boundary curve of length $\varepsilon$. Denote any surface obtained this way $\tilde{S}$. Now consider any sphere $\hat{S}$ with $k$ cusps and one boundary geodesic of length $\varepsilon$. We can now glue $\tilde{S}$ to $\hat{S}$ along their boundary geodesics of length $\varepsilon$ (the twist parameter is unimportant). For any $\varepsilon$, one can obtain such a sphere with $n$ cusps. Using the collar theorem, we now choose a very small $\varepsilon$ so that any reasonable pants decomposition contains the geodesic of length $\varepsilon$. (Note that the value of $\varepsilon$ is, a priori, genus dependent, and implicitly we've used the fact that an upper-bound for $\BB_{0,n}$ exists.) Denote any sphere obtained this way by $S$. Note that ${\mathcal B}(S) > {\mathcal B}(S_{p,0}).$ By the proposition above, we thus have a sphere $S$ with $n$ cusps such that any pants decomposition contains a curve of length $8p \arcsinh 1$. But because $p> \sqrt{\frac{n-4}{2}}-1$, we obtain the desired inequality.

\end{proof}

Using the same method as above, only now using boundary curves of length $\ell\geq 0$, we can use the proposition to find a lower bound on $\BB_{0,n,\ell}$.

\begin{corollary} 
For  $n>4$, Bers' constants $\BB_{0,n,\ell}$ satisfy the following inequality:
$$
\BB_{0,n,\ell} > 8 \arcsinh\sqrt{\cosh \frac{\ell}{8}} (\sqrt{\frac{n-4}{2}}-1)> \frac{\sqrt{n} \ell}{2\sqrt{2}}.
$$
\end{corollary}

\begin{proof}
The first inequality is the one obtained in the proof of the previous corollary when replacing the surfaces $\tilde{S}$ and $\hat{S}$ with surfaces where all cusps are replaced by boundary geodesics of length $\ell$. The second inequality is just a standard (and very crude) algebraic manipulation and of course is only of interest when $\ell$ is strictly positive.  We include it only to give the crude order of growth of our examples in order to show that the crude asymptotics of theorem \ref{thm:boundarysphere} are correct.
\end{proof}

As a third application of the proposition, we shall now derive a lower bound on the Bers' constant for closed hyperbolic surfaces. 

\begin{corollary}
Bers' constants $\BB_g$ satisfies $\BB_g > 4\arcsinh 1 (\sqrt{g-2}-1).$
\end{corollary}

\begin{proof}
Here we mimic the strategy of the previous corollary. We begin by noticing that if we use the example from proposition \ref{prop:hairy} with a choice of $\ell>0$ and gluing together the geodesics pairwise (neither the twist parameter or the choice of pairing matters), one obtains a closed surface of genus $p^2 + 2$.\\

We begin by writing $g= p^2+2+k$ with $g< (p+1)^2$ and $k\geq 0$. Thus $p>\sqrt{g-2}-1$.\\

As in the previous corollary, if $k=0$, then perform the above construction to obtain a surface of genus $g^2+2$ choosing $\ell$ very small so that, by the collar theorem again, any reasonably short pants decomposition contains the geodesics of length $\ell$.\\

If $k>0$, then perform the above the construction, only this time leave two curves of $S_{p,\ell}$ as boundary curves. Denote the surface thus obtained $\tilde{S}$ and note that it is of signature $(p^2+1,2)$. Now take any surface of signature $(k,2)$ with boundary lengths both equal to $\ell$ (say $\hat{S}$). Paste the two surfaces $\tilde{S}$ and $\hat{S}$ together along their boundary curves (once again, this can be done in anyway) to obtain a (closed) surface $S$ of genus $p^2+2+k=g$. Using proposition \ref{prop:hairy} we can conclude:
$$
{\mathcal B}(S)> {\mathcal B}(S_{p,\ell})> 8 p \arcsinh 1 > 8 \arcsinh 1 (\sqrt{g-2}-1).
$$
\end{proof}

Note that this lower bound is an improvement on the lower bound obtained by Buser in \cite{bubook}, although of course both grow like the square root of the genus. The constant in front of the square root in \cite{bubook} is $\sqrt{6} \approx 2.4$ whereas here the constant is $8 \arcsinh 1 \approx 7 $.\\

Finally, we are interested in deriving a lower bound for Bers' constants of hyperelliptic surfaces. Once again, this follows as a corollary of the above constructions.

\begin{corollary}
Bers' constants for closed hyperelliptic surfaces satisfy
$$
\BB^{hyp}_g > 4 \arcsinh 1 (\sqrt{\frac{g-3}{2}}-1).
$$
\end{corollary}

\begin{proof}
For each $g$, consider $\ell$  that by the collar theorem any short pants decomposition of a surface of genus $g$ contains all geodesics of length $\ell$. Consider a sphere $S_{g+1,\ell}$ as constructed above with $g+1$ geodesics of length $\ell$. On each boundary geodesic of $S_{g,\ell}$, consider two diametrically opposite points and glue together the arcs between them to obtain a sphere with $2g+2$ cone points of angle $\pi$. By considering the hyperelliptic double of the sphere thus obtained, one obtains a hyperelliptic closed surface $S$ of genus $g$. By construction, the surface $S$ has $g+1$ simple closed geodesics of length $\ell$ that together separate $S$ into two copies of $S_{g+1,\ell}$. By the above discussion, any short pants decomposition of $S$ contains the geodesics of length $\ell$. The remainder of the pants decomposition consists of a pants decomposition of the two copies of $S_{g+1,\ell}$ and we know that any pants decomposition of $S_{g+1,\ell}$ has length at least $4 \arcsinh \,1 \sqrt{\cosh \frac{\ell}{8}} (\sqrt{\frac{g-3}{2}}-1)$. Thus any pants decomposition of $S$ has length at least $4 \arcsinh \,1(\sqrt{\frac{g-3}{2}}-1)$ and the corollary follows.
\end{proof}


\bibliographystyle{plain}
\def\cprime{$'$}

\end{document}